\documentclass[12pt,a4paper,final]{amsproc}
\usepackage{hyperref}

\newtheorem{theorem}{\bf Theorem}[section]
\newtheorem{lemma}[theorem]{\bf Lemma}
\newtheorem{proposition}[theorem]{\bf Proposition}
\newtheorem{corollary}[theorem]{\bf Corollary}

\author[C. Acciarri]{Cristina Acciarri}
\address{Department of Mathematics, University of Bras\'ilia,~70910-900 Bras\'ilia DF, Brazil}
\email{acciarricristina@yahoo.it}

\author[D. Silveira ]{Danilo Silveira}
\address{Department of Mathematics, Federal University of Goi\'as,~75704-020 Catal\~ao GO, Brazil}
\email{sancaodanilo@ufg.br}


\thanks{This work was supported by the Conselho Nacional de Desenvolvimento Cient\'{\i}fico e Tecnol\'ogico (CNPq),  and Funda\c c\~ao de Apoio \`a Pesquisa do Distrito Federal (FAPDF), Brazil.}

\title[Engel-like conditions in fixed points of automorphisms]{Engel-like conditions in fixed points \\of automorphisms of profinite groups}

\begin{document}

\begin{abstract} Let $q$ be a prime  and $A$ an elementary abelian $q$-group acting as a coprime  group of automorphisms on a profinite group $G$. 
	
We show that if $A$ is of order $q^2$ and  some power of each element in $C_G(a)$ is Engel in $G$ for any $a\in A^{\#}$, then $G$ is locally virtually nilpotent.

Assuming  that $A$ is of order $q^3$ we prove that if some power of each element in $C_G(a)$ is Engel in $C_G(a)$ for any $a\in A^{\#}$, then $G$ is locally virtually nilpotent.
	
Some  analogues consequences of  quantitative nature  for finite groups are also obtained. 
\end{abstract}

\maketitle
\section{Introduction}
A profinite group is a topological group that is isomorphic to an inverse limit of finite groups. In the context of profinite groups all the usual concepts of groups theory are interpreted topologically. In particular, by a subgroup of a profinite group we always mean a closed subgroup and   a subgroup is said to be generated by a set $S$ if it is topologically generated by $S$. See, for example, \cite{W} for these and other properties of profinite groups.
Many remarkable results on profinite groups were deduced using  the Lie-theoretic machinery developed for  the solution of the restricted Burnside problem \cite{Z0,Z1,zenew}. For instance, using Wilson's reduction theorem \cite{W2}, Zelmanov proved that a profinite group is locally finite if and only if it is torsion \cite{zelmanov}.
 Recall that a  group $G$ is said to have a certain property locally if any finitely generated subgroup of $G$ possesses that property.  We say that a group $G$ is torsion if all of its elements have finite order.  
 
Another  well-known result of Wilson and Zelmanov \cite[Theorem 5]{WZ}  tells us that a profinite group is locally nilpotent if and only if it is Engel. If  $x,y$ are elements of a (possibly infinite) group $G$, the commutators $[x,_n y]$ are defined inductively by the rule
$$[x,_0 y]=x,\quad [x,_n y]=[[x,_{n-1} y],y]\quad \text{for all}\, n\geq 1.$$
Recall that an element $x$ is called a (left) Engel element if for any $g\in G$ there exists $n$, depending on $x$ and $g$, such that $[g,_n x]=1$.  A group $G$ is called Engel if all elements of $G$ are Engel. The element $x$ is called a (left) $n$-Engel element if for any $g\in G$ we have $[g,_n x]=1$. The group $G$ is $n$-Engel if all elements of $G$ are $n$-Engel.

Later on, in \cite{BS}, Bastos and Shumyatsky considered profinite groups with Engel-like conditions. They showed in \cite[Theorem 1.1]{BS} that  if $G$ is a profinite group in which for every element $x\in G$  there exists a natural number $q = q(x)$ such that $x^q$ is Engel, then $G$ is locally virtually nilpotent. We recall that a profinite group posses a certain property virtually  if it has an open subgroup with that property.  Note that the  previous result can be  viewed as a common generalization of both the  Wilson--Zelmanov results on profinite groups stated above.

By an automorphism of a profinite group we mean a continuous automorphism. We say that a   group $A$  acts  on a profinite group $G$ coprimely  if $A$ has finite order while $G$ is an inverse limit of finite groups whose orders are relatively prime to the order of $A$.  
In the literature   there are many well-known results  showing that if  $A$ is a  finite group acting on a finite group $G$ in such a manner that $(|A|,|G|)=1$, then  the structure of the centralizer $C_G(A)$ (the fixed-point subgroup) of $A$ has a strong influence over the structure of $G$ (see for instance \cite{CPD, KS,Pavel-law,shusa}).  A similar phenomenon  holds in the realm of profinite groups:  we see that imposing restrictions on centralizers of coprime automorphisms result in very specific structures for the group $G$. Given an automorphism $a$ of a profinite group $G$, we denote by $C_G(a)$ the centralizer of $a$ in $G$, that is, the subgroup formed by the elements  fixed under $a$. In particular, the following theorems were  established in \cite[Theorem 1.1]{eu} and \cite[Theorem B2]{CPD}, respectively. 
 
\begin{theorem}\label{periodic} 
Let $q$ be a prime and  $A$   an elementary abelian $q$-group of order at least $q^2$ acting  coprimely on a profinite group $G$. Assume that the centralizer $C_{G}(a)$ is torsion for each $a\in A^{\#}$. Then $G$ is locally finite.
\end{theorem}

\begin{theorem}\label{engel}
Let $q$ be a prime and $A$ be an elementary abelian $q$-group of order at least $q^2$ acting  coprimely on a profinite group $G$. Assume that  all elements in $C_{G}(a)$ are Engel in $G$ for each $a\in A^{\#}$. Then $G$ is locally nilpotent.
\end{theorem}

Here and throughout the paper $A^{\#}$ denotes the set of nontrivial elements of $A$. The proofs of the above results involve a number of deep ideas. In particular, Lie theoretical results of Zelmanov \cite{Z0,Z1,zenew} obtained in his solution of the restricted Burnside problem are combined with a criteria for a pro-$p$ group to be $p$-adic analytic in terms of the associated Lie algebra due to Lazard \cite{L}, and 
a theorem of Bahturin and Zaicev \cite{BZ} on Lie algebras admitting a soluble group of automorphisms whose fixed-point subalgebra satisfies a polynomial identity. Moreover Theorems \ref{periodic} and \ref{engel} rely heavily on
Zelmanov's theorem about  local finiteness of torsion profinite groups and on  the Wilson-Zelmanov result on local nilpotency of Engel profinite groups, respectively.

In the present paper  we consider  profinite groups admitting an action by an elementary abelian group under which the centralizers of automorphisms  satisfy the property that some power of any element is Engel.  Our first goal is to establish the following result.

\begin{theorem}\label{main}
Let $q$ be a prime and $A$ an elementary abelian group of order $q^2$. Suppose that $A$ acts  coprimely on a profinite group $G$ and assume that some power of each element in $C_G(a)$ is Engel in $G$ for any $a\in A^{\#}$. Then $G$ is locally virtually nilpotent. 
\end{theorem}

Using  Theorem \ref{main} in combination with the positive solution of the restricted Burnside problem \cite{Z0,Z1,zenew} the following  quantitative result for finite groups can be obtained. 
\begin{corollary}\label{coro1}
	Let $m, d$ be integers, $q$  a prime and $A$ an elementary abelian group of order $q^2$. Suppose that $A$ acts coprimely on a $m$-generated finite group $G$ and assume that all $d$-th powers of elements in $C_{G}(a)$ are $n$-Engel in $G$ for each $a\in A^{\#}$. Then there exist positive integers $e$ and $c$, depending only on $m,n,q$ and $d$, such that $G$ has a normal subgroup $N$ with nilpotency class at most $c$ and $|G/N|$ is at most $e$. 
\end{corollary}


If, in Theorem \ref{engel}, we relax the hypothesis that every element of $C_{G}(a)$ is Engel in $G$ and require instead that every element of $C_{G}(a)$ is Engel in $C_{G}(a)$, we  see that the result is no longer true. Indeed, an example of a finite non-nilpotent group $G$ admitting  an action of a non-cyclic group $A$ of order four  such that $C_G(a)$ is abelian for each $a\in A^{\#}$ can be found for instance in \cite{PS-CA3}. On the other hand, in \cite{PS-CA3}, the authors proved that if  $A$ is an elementary abelian $q$-group of order at least $q^3$, with $q$ a prime, acting coprimely on a profinite group $G$ in such a manner  that $C_{G}(a)$ is locally nilpotent for each $a\in A^{\#}$, then $G$ is locally nilpotent.
Another purpose of the present paper is to  establish the following related result.

\begin{theorem}\label{main2}
Let $q$ be a prime and $A$ an elementary abelian group of order $q^3$. Suppose that $A$ acts coprimely on a profinite group $G$ and assume that some power of each element in $C_G(a)$ is Engel in $C_G(a)$ for any $a\in A^{\#}$. Then $G$ is locally virtually nilpotent. 
\end{theorem}

Our next result  is  an  analogue  of Corollary \ref{coro1}.
\begin{corollary}\label{coro2}
	Let $m, d$ be integers, $q$  a prime and $A$ an elementary abelian group of order $q^3$. Suppose that $A$ acts coprimely on a $m$-generated finite group $G$ and assume that all $d$-th powers of elements in $C_{G}(a)$ are $n$-Engel in $C_{G}(a)$ for each $a\in A^{\#}$. Then there exist positive integers $e$ and $c$, depending only on $m,n,q$ and $d$, such that $G$ has a normal subgroup $N$ with nilpotency class at most $c$ and $|G/N|$ is at most $e$. 
\end{corollary}

The paper is organized as follows. In Sections 2  we present the Lie-theoretic machinery that will be useful within the proofs.  Section 3 is devoted to proving Theorem \ref{main} and Corollary \ref{coro1},  and in the last section we establish  Theorem \ref{main2} and Corollary \ref{coro2}. 

The notation is standard. Throughout the paper we use, without special references, the well-known properties of coprime actions (see for example \cite[Lemma 3.2]{PSprofinite}). 
 
If $\alpha$ is a coprime automorphism of a profinite group $G$, then \linebreak $C_{G/N}(\alpha)=C_G(\alpha)N/N$ for any $\alpha$-invariant normal subgroup $N$.
 
If $A$ is a noncyclic elementary abelian group acting coprimely on a profinite group $G$, then $G$ is generated by the subgroups $C_G(B)$, where $A/B$ is cyclic.

 \section{Associated Lie algebras}
Let $L$ be a Lie algebra over a field $K$ and  $X$  a subset of $L$. By a commutator in elements of $X$ we mean any element of $L$ that can be obtained as a Lie product of elements of $X$ with some system of brackets. If $x_1,\ldots,x_k,x, y$ are elements of $L$, we define inductively 
$$[x_1]=x_1; [x_1,\ldots,x_k]=[[x_1,\ldots,x_{k-1}],x_k]$$
and 
$[x,_0y]=x; [x,_my]=[[x,_{m-1}y],y],$ for all positive integers $k,m$.  
As usual, we say that an element $a\in L$ is ad-nilpotent if there exists a positive integer $n$ such that $[x,_na]=0$ for all $x\in L$. If $n$ is the least integer with the above property, then we say that $a$ is ad-nilpotent of index $n$. Denote by $F$ the free Lie algebra over $K$ on countably many free generators $x_1,x_2,\ldots$. Let $f=f(x_1,x_2,\ldots,x_n)$ be a non-zero element of $F$. The algebra $L$ is said to satisfy the identity $f =0$ if $f(l_1,l_2,\ldots,l_n) = 0$ for any $l_1,l_2,\ldots,l_n\in L$. In this case we say also that $L$ satisfies a PI (polynomial identity) or that $L$ is a PI-algebra.

The next theorem represents the most general form of the Lie-theoretical part of the solution of the restricted Burnside problem. It was announced by Zelmanov in \cite{Z1}. A detailed proof can be found in \cite{zenew}.

\begin{theorem}\label{Z1992}
Let $L$ be a Lie algebra over a field and suppose that $L$ satisfies a PI. If $L$ can be generated by a finite set $X$ such that every commutator in elements of $X$ is ad-nilpotent, then $L$ is nilpotent.
\end{theorem}

An important criterion for a Lie algebra to satisfy a PI is provided by the next theorem, which was proved by Bahturin and Zaicev for soluble groups of automorphisms \cite{BZ} and later extended by Linchenko to the general case \cite{l}. 
\begin{theorem}\label{blz}
Let $L$ be a Lie algebra over a field $K$. Assume that a finite group $A$ acts on $L$ by automorphisms in such a manner that $C_L(A)$ satisfies a PI. Assume further that the characteristic of $K$ is either $0$ or prime to the order of $A$. Then $L$ satisfies a PI.
\end{theorem}
We use the centralizer notation for the fixed point subalgebra $C_L(A)$ of a group of automorphisms $A$ of $L$.

Another useful result, whose proof can be found in \cite[Lemma 5]{KS}, is the following. 

\begin{lemma}\label{lemmanovo}
Let $L$ be a Lie algebra and $H$ a subalgebra of $L$ generated by $m$ elements $h_{1},\ldots, h_{m}$ such that all commutators in the generators $ h_i$ are  ad-nilpotent in $L$. If $H$ is nilpotent, then  we have  $[L,\underbrace{H,\ldots, H}_u]=0$ for some  number $u$.
\end{lemma}

Let $G$ be a (profinite) group. A series of subgroups $$G=G_1\geq G_2\geq\dots\eqno{(*)}$$ is called an $N$-series if it satisfies $[G_i,G_j]\leq G_{i+j}$ for all $i,j\geq 1$. Here and throughout the paper when dealing with a profinite group we consider only closed subgroups. Obviously any $N$-series is central, i.e. $G_i/G_{i+1}\leq Z(G/G_{i+1})$ for any $i$.  Let $p$ be a prime. An $N$-series is called $N_p$-series if $G_i^p\leq G_{pi}$ for all $i$. Given an $N$-series $(*)$, let $L^*(G)$ be the direct sum of the abelian groups $L_i^*=G_i/G_{i+1}$, written additively. Commutation in $G$ induces a binary operation $[,]$ in $L^*(G)$. For homogeneous elements $xG_{i+1}\in L_i^*,yG_{j+1}\in L_j^*$ the operation is defined by $$[xG_{i+1},yG_{j+1}]=[x,y]G_{i+j+1}\in L_{i+j}^*$$ and extended to arbitrary elements of $L^*(G)$ by linearity. It is easy to check that the operation is well-defined and that $L^*(G)$ with the operations $+$ and $[,]$ is a Lie ring. If all quotients $G_i/G_{i+1}$ of an $N$-series $(*)$ have prime exponent $p$ then $L^*(G)$ can be viewed as a Lie algebra over $\mathbb F_p$, the field with $p$ elements. In the important  case where the series $(*)$ is the $p$-dimension central series (also known under the name of Zassenhaus-Jennings-Lazard series) of $G$ we write $D_i=D_i(G)$ for  the $i$-th term of the series of $G$, $L(G)$ for the corresponding associated Lie algebra over the field with $p$ elements  and  $L_p(G)$ for the subalgebra generated by the first homogeneous component $D_1/D_2$ in $L(G)$. Observe that the  $p$-dimension central series is an $N_p$-series (see \cite[p.\ 250]{Huppert2} for details).

Any automorphism of $G$ in the natural way induces an automorphism of $L^*(G)$. If $G$ is profinite and $\alpha$ is a coprime automorphism of $G$, then the subalgebra of fixed points of $\alpha$ in $L^*(G)$ is isomorphic to the Lie algebra associated with the group $C_G(\alpha)$ via the series formed by intersections of $C_G(\alpha)$ with the terms of the series $(*)$ (see \cite{aaaa} for more details).

Given  an $N_p$-series $(*)$ of $G$, let $x\in G$ and let $i=i(x)$ be the largest positive integer such that $x\in G_i$. We denote by $x^*$ the element $xG_{i+1}\in L^{*}(G)$.  We now quote some results providing sufficient conditions for $x^*$ to be ad-nilpotent. The first lemma is due to Lazard (see \cite[p.\ 131]{L2}).

\begin{lemma}\label{lazard-ad}
For any $x\in G$ we have $(adx^{*})^p=ad(x^p)^*$. In particular, if $x$ is of finite order $t$, then $x^*$ is ad-nilpotent of index at most $t.$
\end{lemma}

The next result  essentially is due to Wilson and Zelmanov since it follows from the proof of  \cite[Lemma in Section 3]{WZ}.
\begin{lemma}\label{leWZ} Let $x$ be an Engel element of a profinite group $G$. Then $x^*$ is ad-nilpotent.
\end{lemma}

Combining Lemmas \ref{lazard-ad} and \ref{leWZ} it is easy to deduce the following result.

\begin{lemma}\label{lemaconseque}
Let $x$ be an element of a profinite group $G$ for which there exists a positive integer $d$ such that $x^d$ is Engel. Then $x^*$ is ad-nilpotent.
\end{lemma}

A group $G$ is said to satisfy a coset identity if there is a nontrivial group word $w=w(x_1,x_2,\dots,x_k)$ and cosets $g_1H,g_2H,\dots,g_kH$ of a subgroup $H$ of $G$ of finite index such that   $w(g_1h_1,g_2h_2,\dots,g_kh_k)=1$ for all  $h_1,h_2,\dots,h_k\in H$; in this case we can also say that 
 the law $w\equiv1$ is satisfied on the cosets $g_1H,g_2H,\dots,g_kH$.  
 In \cite[Theorem 1]{WZ} Wilson and Zelmanov proved the following theorem.

\begin{theorem}\label{WZ-cosets}
If a profinite group $G$ has an open  subgroup $H$
and elements $g_1,\ldots,g_k$ such that a law $w\equiv 1$ is satisfied on the cosets $a_1H,\ldots,a_kH$, then for each prime $p$ the Lie algebra $L_p(G)$ satisfies a PI.
\end{theorem}

\section{Proof of Theorem \ref{main} and Corollary \ref{coro1} }
We start this section by proving the following  useful result. 

\begin{lemma}\label{lpi}
Let $p$ be a prime. Suppose that a finite group $A$ acts coprimely on a profinite group $G$.  Assume that some power of each element in $C_G(A)$ is Engel in $C_G(A)$ for any $a\in A^{\#}$. Then $L_p(G)$ satisfies a multilinear  polynomial identity.
\end{lemma}
\begin{proof} Let $L=L_p(G)$.
In view of Theorem \ref{blz} it is sufficient to show that $C_L(A)$ satisfies a polynomial identity. We know that $C_L(A)$ is isomorphic to the Lie algebra associated with the central series of $C_G(A)$ obtained by intersecting $C_G(A)$ with the $p$-dimension central series of $G$. 
For each  pair $i,j$ of positive integers we set $$S_{ij}=\{(g,h)\in C_G(A)\times C_G(A):[g,_ih^j]=1\}.$$
Since the sets $S_{ij}$ are closed in $C_G(A)\times C_G(A)$ and their union is $C_G(A)\times C_G(A)$, by Baire's category theorem \cite[p.\ 200]{Baire} at least one of these sets has a non-empty interior. Therefore, we can find an open subgroup $H$ in $C_G(A)$, elements $u,v\in C_G(A)$ and  integers $n,d$ such that the identity $[x,_ny^d]\equiv 1$ is satisfied on the cosets $uH, vH$. Thus,  Theorem \ref{WZ-cosets} applies and $C_L(A)$ satisfies a  polynomial identity, as desired. 
\end{proof}

Next, we will prove  Theorem \ref{main} under the additional hypothesis that  $G$ is a pro-$p$ group.

\begin{proposition}\label{casepro-p}
Let $G$ be a pro-$p$ group satisfying the hypothesis of Theorem \ref{main}. Then $G$ is locally virtually nilpotent.
\end{proposition}
\begin{proof}

Since every finite subset of $G$ is contained in a finitely generated $A$-invariant closed subgroup, we may assume that $G$ is finitely generated.  Then, of course, it will be sufficient to show that $G$ is virtually nilpotent. 

Let $H$ be the subgroup generated by all Engel elements  in $G$. Note that $H$ is a normal $A$-invariant subgroup of $G$. Thus, for each $a\in A^{\#}$, we have  $C_{G/H}(a)=C_G(a)H/H$, which  is a torsion subgroup. By Theorem \ref{periodic}, $G/H$ is finite, and so $H$ is open. Since $G$ is finitely generated,  \cite[Proposition 4.3.1]{W} implies that $H$ is finitely  generated, as well. We claim that $H$ is nilpotent.

Indeed, we denote by $D_j=D_j(H)$ the terms of the $p$-dimension central series of $H$. Let $L=L_p(H)$ be the Lie algebra associated with the group $H$ and 
$L_j=L\cap(D_j/D_{j+1})$. Thus, $L=\oplus_{j\geq 1}L_j$. 
The group $A$ naturally acts on $L$. Let $A_1,\ldots,A_{q+1}$ be the distinct maximal subgroups of $A$. Set $L_{ij}=C_{L_j}(A_i)$. We know that any $A$-invariant subgroup is generated by the centralizers of $A_i$. Therefore for any $j$ we have $$L_j=\sum_{i=1}^{q+1}L_{ij}.$$
Further, for any $l\in L_{ij}$ there exists an element $x\in D_j\cap C_H(A_i)$ such that $l=xD_{j+1}$. 
Hence, there exists a positive integer $d$ such that $x^d$ is
Engel in $ H $. It follows from Lemma \ref{lemaconseque} that $l$ is ad-nilpotent in $L$. Thus,

\begin{equation}\label{le1}
\mbox{any element in}\  L_{ij}  \ \mbox{is ad-nilpotent in}\  L.
\end{equation}

Let $\omega$ be a primitive $q$th root of unity and $\overline{L}=L\otimes\mathbb{F}_p[\omega]$. Here $\mathbb{F}_p$ stands for the field with $p$ elements. We can view $\overline{L}$ both as a Lie algebra over $\mathbb{F}_p$ and as that over $\mathbb{F}_p[\omega]$. It is natural to identity $L$ with the subalgebra $L\otimes 1$ of $\overline{L}$. We note that if an element $x\in L$ is ad-nilpotent of index $r$, say, then the correspondent element $x\otimes1$ is ad-nilpotent in $\overline{L}$ of the same index $r.$

Put $\overline{L}_j=L_j\otimes\mathbb{F}_p[\omega]$. Then $\overline{L}=\left<\overline{L}_1\right>$ and $\overline{L}$ is the direct sum of the homogeneous components $\overline{L}_j$. The group $A$ naturally acts on $\overline{L}$ and we have $\overline{L}_{ij}=C_{\overline{L}_j}(A_i)$, where $\overline{L}_{ij}=L_{ij}\otimes\mathbb{F}_p[\omega]$. Let us show that
\begin{equation}\label{le2}
\mbox{any element}\ y\in\overline{L}_{ij}\ \mbox{is ad-nilpotent in}\ \overline{L}.
\end{equation}
Since $\overline{L}_{ij}=L_{ij}\otimes\mathbb{F}_p[\omega]$, we can write
$$y=x_0+\omega x_1+\omega^2 x_2\cdots+\omega^{q-2}x_{q-2}$$
for suitable $x_0, x_1, x_2,\ldots, x_{q-2}\in L_{ij}$. In view of (\ref{le1}) it is easy to see that each of the summands $\omega^sx_s$ is ad-nilpotent in $\overline{L}$. Let $J$ be the subalgebra of $\overline{L}$ generated by $x_0,$ $ \omega x_1,\omega^2 x_2,\ldots, \omega^{q-2}x_{q-2}$. We wish to show that $J$ is nilpotent. 

Note that $J\subseteq C_{\overline{L}}(A_i)$. A commutator of weight $t$ in the generators of $J$ has form $\omega^\alpha x$ for some $x$ that belongs to $L_{im}$, where $m=tj$. By (\ref{le1}) the element $x$ is ad-nilpotent in $\overline{L}$ and so such a commutator must be ad-nilpotent.  By Lemma \ref{lpi} $L$ satisfies a multilinear polynomial identity. The multilinear identity is also satisfied in $\overline{L}$ and so it is satisfied in $J$, since $J\subseteq C_{\overline{L}}(A_i)$. Hence, by Theorem \ref{Z1992} $J$ is nilpotent. Now, applying Lemma \ref{lemmanovo}, we get that  there exists some positive integer $u$ such that $[\overline{L}, \underbrace{J,\ldots,J}_u ]=0$. This proves (\ref{le2}).

Since $A$ is abelian and the ground field is now a splitting field for $A$, every component $\overline{L}_j$ decomposes in the direct sum of common eigenspaces for $A$. In particular,  $\overline{L}_1$ is spanned by finitely many common eigenvectors for $A$, since $H$ is a finitely generated pro-$p$ group. Hence, $\overline{L}$ is generated by finitely many common eigenvectors for $A$ from $\overline{L}_1$. Since $A$ is noncyclic every common eigenvector is contained in the centralizer $C_{\overline{L}}(A_i)$ for some $i\leq q+1$.

We also note that any commutator in common eigenvectors is again a common eigenvector for $A$. Therefore, if $l_1, l_{2},\ldots\in \overline{L}_1$ are common eigenvectors for $A$ generating $\overline{L}$, then any commutator in those generators belongs to some $\overline{L}_{ij}$ and so, by (\ref{le2}), is ad-nilpotent.

As we have mentioned earlier, $\overline{L}$ satisfies a polynomial identity. It follows from Theorem \ref{Z1992} that $\overline{L}$ is nilpotent. Since $L$ embeds into $\overline{L}$, we  deduce that $L$ is nilpotent as well. 

According to Lazard \cite{L} the nilpotency of $L$ is equivalent to $H$ being $p$-adic analytic (for details see  \cite[A.1 in Appendice and  Sections 3.1 and 3.4 in Ch.\ III]{L} or  \cite[1.(k) and 1.(o) in Interlude A]{GA}). It  follows from \cite[7.19 Theorem]{GA} that $H$ admits a faithful linear representation over the field of $p$-adic numbers. 

Since  $H$ is a finitely generated pro-$p$ group and can be  generated by  Engel elements,  by using an inverse limit argument combined  with the Burnside Basis Theorem \cite[5.3.2]{Rob}, we see that $H$ is  generated by finitely many Engel elements. A result of Gruenberg \cite[Theorem 0]{G} says that in a linear group the Hirsch-Plotkin radical coincides with the set of Engel elements. Then it follows  that $H$ is nilpotent, as claimed. This concludes the proof. 
\end{proof}

As usual, for a profinite group $G$ we denote by $\pi(G)$ the set of prime divisors of the orders of finite continuous homomorphic images of $G$. We say that $G$ is a $\pi$-group if $\pi(G)\subseteq\pi$ and $G$ is a $\pi'$-group if $\pi(G)\cap\pi=\emptyset$. If $m$ is an integer, we denote by $\pi(m)$ the set of prime divisors of $m$. If $\pi$ is a set of primes, we denote by $O_\pi(G)$ the maximal normal $\pi$-subgroup of $G$ and by $O_{\pi'}(G)$ the maximal normal $\pi'$-subgroup. 

Now, we are ready to deal with the proof of Theorem \ref{main}.

\begin{proof}[{\bf Proof of Theorem \ref{main}}] It will be convenient first to prove the theorem under the additional hypothesis that $G$ is pronilpotent. Thereby $G$ is the Cartesian product of its Sylow subgroups. Choose $a\in A^{\#}$. For each pair $i,j$ of positive integers we set $$S_{ij}=\{(x,y)\in G\times C_G(a):[x,_iy^j]=1\}.$$
Arguing as in the proof of Lemma \ref{lpi} we deduce that there exist an open normal subgroup $H$ in $G$, elements $u\in G, v\in C_G(a)$ and  positive integers $n,d$ such that $[ul,_n(vk)^d]=1$ for any $l\in H$ and any $k\in H\cap C_G(a)$. 

Let $[G:H]=m$ and let  $\pi_1=\pi(m)$ be the set of primes dividing $m$. Denote  $O_{\pi'_1}(G)$ by $K$ and write $J$ for the Sylow subgroups of $G$ corresponding to the primes that belong to $\pi_1$. Since $G=J\times K=JH$ we deduce that $[x,_ny^d]=1$, for all $x\in K$ and $y\in C_K(a)$.  Now set $\pi_2=\pi(d)$ and $\pi=\pi_1\cup\pi_2$. Denote $O_{\pi'}(G)$ by $T$. Since by construction $(p,d)=1$ for each prime $p\in \pi(T)$, it follows that  every element $y$ in $C_T(a)$ is $n$-Engel in $T$.

Of course, the set $\pi$  and the integer $n$ depend only on the choice of $a\in A^{\#}$, so strictly speaking they should be denoted by $\pi_a$ and $n_a$, respectively. We choose such $\pi_a$ and $n_a$ for any $a\in A^{\#}$. Set $\pi_0=\cup_{a\in A^{\#}}\pi_a$, $n={\rm max}\{n_a:a\in A^{\#}\}$ and $R=O_{\pi_0'}(G)$. 

By construction we see that, for each $a\in A^{\#}$, every element of $C_R(a)$ is $n$-Engel in $R$. Using an  inverse limit argument we deduce from \cite[Theorem 1.2]{shusa} that $R$ is $s$-Engel for some integer $s$. Thus  \cite[Theorem 5]{WZ} implies  that $R$ is locally nilpotent.
Let $p_1,\ldots, p_r$ be the finitely many primes in $\pi$ and $P_1,\ldots, P_r$ be the corresponding Sylow subgroups of $G$. Then $G=P_1\times\ldots\times P_r\times R$ and therefore it is sufficient to show that each subgroup $P_i$ is locally virtually nilpotent. But, this is immediate from Proposition \ref{casepro-p}. 
This proves the result in the particular case where $G$ is pronilpotent.

Let us now  drop the assumption that $G$ is pronilpotent. Without loss of generality we can assume that $G$ is finitely generated. Set $K$ be the closure of the  subgroup of $G$ generated by all Engel elements in $G$. Note that $K$ is a normal $A$-invariant subgroup. Since 
$C_{G/K}(a)=C_G(a)K/K$,
for any $a\in A^{\#}$, in particular we know that each centralizer is torsion. Now Theorem \ref{periodic} implies that  $G/K$ is finite and, therefore, $K$ is finitely generated. By  Baer's Theorem \cite[12.3.7]{Rob} we deduce  that $K$ is a pronilpotent group.  Hence, using what we showed above, we conclude that $K$ is virtually nilpotent and  this completes the proof.
\end{proof}

We close this section by giving the  proof of Corollary   \ref{coro1}.
\begin{proof}[{\bf Proof of Corollary \ref{coro1}}]
Suppose that the corollary is false. Then, for each pair of  positive integers $i,j$, we can choose a group $G_{ij}$  satisfying the hypothesis of the corollary and having all of  its normal subgroups  either with nilpotency class at  least $i$ or with index in $G_{ij}$ at least $j$ (or both properties). In each group $G_{ij}$, we fix generators $g_1^{ij},\ldots,g_m^{ij}$.

Let $G$ be the Cartesian product of the groups $G_{ij}$, assuming that we use the lexicographic order to construct  the Cartesian product. Note that $G$ is a profinite group admitting a coprime action of $A$ and  such that all $d$-th power of elements in $C_G(a)$ are $n$-Engel in $G$ for each $a\in A^{\#}$. Thus, by Theorem \ref{main}, $G$ is locally virtually nilpotent. 

 In $G$ consider the closed subgroup $D$ generated by $m$ elements
$$g_1=(g_1^{11},g_1^{12},\ldots), \ldots, g_m=(g_m^{11},g_m^{12},\ldots). $$
Thus, $D$ has a open nilpotent normal subgroup $K$ of  class $c$, say. Let $r$ be the index of $K$ in $D$ and observe that both  $c$ and $r$ are numbers  that depend only on $m,n,q$ and $d$. We remark that each of the groups $G_{ij}$ is isomorphic to a finite quotient of  $D$. Thus, each subgroup $G_{ij}^r$ is nilpotent of class at most $c$. Furthermore, by the positive solution of the  restricted Burnside problem \cite{Z0,Z1,zenew}, we know the index of $G_{ij}^r$ in $G_{ij}$  depends only on $m$ and $r$. This leads to a contradiction.
\end{proof}
 
\section{{\bf Proof of Theorem \ref{main2} and Corollary \ref{coro2}}}

Let $F$ denote the free group on free generators $x_1, x_2, \ldots$.  Recall that a positive
word in $X=\{x_1, x_2,\ldots\}$ is any nontrivial element of $F$ not involving the inverses of the $x_i$. A positive (or semigroup) law of a group
$G$ is a nontrivial identity of the form $u \equiv v $ where $u, v$ are positive words in $F$, holding under every substitution of elements of $X$ by elements of $G$. The maximum of lengths of $u$ and $v$ is called the degree of the law $u \equiv v$. 

By a result of Mal'cev \cite{Malcev}  (see also \cite{NT}) a group that is an extension of a nilpotent group by a group of finite exponent satisfies a positive law. More precisely, Mal'cev discovered a positive law $M_c (x, y)$ in two variables and of degree $2^c$ that holds in any nilpotent group of class $c$. Therefore, if $G$ is an extension of a nilpotent group of class $c$ by a group of exponent $e$, then $G$ satisfies the positive law $M_c(x^e,y^e)$. The explicit form of the Mal'cev law will not be required in this paper.

Next result is a profinite version of  \cite[Theorem A]{Pavel-law}. 

\begin{lemma}\label{positivelaw}
Let $q$ be a prime and $A$ an elementary abelian group of order $q^3$. Suppose that $A$ acts coprimely on a profinite group $G$ and assume that $C_G(a)$ satisfies a  positive law  of degree $n$ for each $a\in A^{\#}$. Then $G$ satisfies a positive law as well. 
\end{lemma}
\begin{proof}	
The result follows easily by using an inverse limit argument and noting that, by the proof of \cite[Theorem A]{Pavel-law}, any finite quotient of $G$ over an $A$-invariant open normal subgroup $N$  satisfies the positive law $M_c(x^k,y^k)$, for some positive integers $c$ and $k$ which do not depend on the choice of $N$ but only on $n$ and $q$.	
\end{proof}
We are ready to embark on the proof of Theorem \ref{main2}. First we consider the case where $G$ is a pro-$p$ group.
\begin{proposition}\label{q^3pro-p}
Let $G$ be a pro-$p$ group satisfying the hypothesis of Theorem \ref{main2}. Then $G$ is locally virtually nilpotent.
\end{proposition}
\begin{proof}
	
Since every finite set of $G$ is contained in a finitely generated $A$-invariant closed subgroup, we may assume that $G$ is finitely generated.  It will be sufficient to show that $G$ is virtually nilpotent.

We denote by $D_j=D_j(G)$ the terms of the $p$-dimension central series of $G$. Let $L=L_p(G)$ be the Lie algebra associated with the group $G$ and $L_j=L\cap(D_j/D_{j+1})$. Thus, $L=\oplus_{j\geq 1}L_j$. The group $A$ naturally acts on $L$. 
Let $A_1,\ldots,A_{s}$ be the distinct maximal subgroups of $A$. Since each subgroup $A_i$ is noncyclic we get  $L=\sum_{a\in A_i}C_{L}(a)$, for every $i\leq s.$
Set $L_{ij}=C_{L_j}(A_i)$. Hence for any $j$ we get  $$L_j=\sum_{i=1}^{s}L_{ij}.$$ 
Thus for any $l\in L_{ij}$ there exists an element $x\in D_j\cap C_G(A_i)$ such that $l=xD_{j+1}$. By assumption, some power of $x$ is Engel in $C_G(A_i)\subseteq C_G(a)$, for some $a\in A_i$. It follows from Lemma \ref{lemaconseque} that $l$ is ad-nilpotent in $C_L(a)$ for every $a\in A_i^{\#}$. Since $L=\sum_{a\in A_i}C_{L}(a)$, we deduce that 
any element $l\in L_{ij}$ is ad-nilpotent in $L$. Now, mimicking the argument that we used  in the proof of Proposition \ref{casepro-p}, with only obvious changes,  one can show that $L$ is nilpotent.  We omit  further details.

According to Lazard \cite{L} the nilpotency of $L$ is equivalent to $G$ being $p$-adic analytic.  The Lubotzky-Mann theory \cite{LM} ensures that $G$ has finite rank. Then, each centralizer $C_G(a)$ is finitely generated. Now, applying  \cite[Theorem 1.1]{BS}, we know that all centralizers $C_G(a)$ are  virtually nilpotent. Thus, there exist a $p$-power $k$ and a positive integer $c$ such that,  for each $a\in A^{\#}$, the subgroup   $C_G(a)^k$ has nilpotency class at most $c$.  A result of  Mal'cev \cite{Malcev} implies now  that   all centralizers $C_G(a)$ satisfy the positive law $M_c(x^k,y^k)$, and so,  Lemma \ref{positivelaw}  yields that $G$ satisfies a positive law too. In accordance with the theorem by  Burns, Macedo{\'n}ska and Medvedev \cite{BMM} the group $G$  is an extension of a nilpotent group $N$ by a group of finite exponent. Finally,  it follows from \cite[Theorem 1]{zelmanov} that $G/N$ is finite and this  completes the proof.
\end{proof}

Recall that the Fitting subgroup of a finite group $H$ is the unique largest normal nilpotent subgroup of $H$, which will be denoted by $F(H)$. 
Similarly, for any profinite group $G$, we will denote by $F(G)$ the (unique) largest normal pronilpotent subgroup of $G$. We remark that any Engel element in a profinite group $G$ belongs to $F(G)$. Indeed, let $K$ be the closed subgroup of $G$ generated by the set of all Engel elements in $G$. By  Baer's theorem \cite[12.3.7]{Rob},  the image of $K$ in every finite quotient of $G$ is nilpotent. Since $K$ is normal in $G$, we see that $K$ is pronilpotent, and so, in particular  contained in $F(G)$. 

\begin{proof}[{\bf Proof of Theorem \ref{main2}}]

 
Let $A_1, \ldots, A_s$ be the distinct maximal subgroups of $A$. Fix $A_i$ and take $x\in C_G(A_i)$. Note that $C_G(A_i)\subseteq C_G(a)$ for any $a\in A_i$. So, by assumption, there exists a positive integer $u_a$, depending on $a$, such that $x^{u_a}\in F(C_G(a))$. Then there exist positive integers $u_1,\ldots,u_{q^2}$ such that 
$x^{u_1\cdots u_{q^2}}\in F(C_G(a)),$\ \  for all $a\in A_i.$

Let $N$ be any $A$-invariant open normal subgroup of $G$. By \cite[Lemma 2.6]{Pavel-law}, we know that the image of $x^{u_1\ldots u_{q^2}}$ in the finite quotient $G/N$ belongs to $\bigcap_{a\in A_i^{\#}} F(C_{G/N}(a))\leq F(G/N)$. Thus, the element $x^{u_1\ldots u_{q^2}}$ belongs to  $F(G)$. Since $A_i$ and $x$ were chosen arbitrarily we can repeat the argument for any $x\in C_G(A_i)$ and  $i\in\{1,\ldots, s\}$. Then the images of the centralizers $C_G(A_i)$ in the quotient group $\overline{G}=G/F(G)$ are all torsion subgroups. 

Let  $a\in A^{\#}$ and consider $K=C_{\overline{G}}(a)$. The group $A$  naturally acts on $ K $ inducing an automorphism group $ A_0 $. Further, for any $ \alpha\in A_0 $ the centralizer $ C_K(\alpha) $ is exactly  $ C_{\overline{G}}(A_i), $ where $ A_i=\langle a,a_1\rangle $ and $ a_1 $ is the element of $ A $ that induces $\alpha$ (in the action of $ A $ on $ K $). We claim that  $C_{K}(\alpha)$ is torsion, for every $ \alpha\in A_0. $ Indeed, if $A_0$ has order $ q^2 $, then it follows from Theorem \ref{periodic} that $K$ is torsion.  If the order of $ A_0 $ is  less than $ q^2 $, then  $ K=C_{\overline{G}}(a)\leq C_{\overline{G}}(A_i) $ and so,  $ K $ is torsion as well. Thus, we get that  $C_{\overline{G}}(a) $ is torsion for any $ a\in A^{\#} $. Applying again Theorem \ref{periodic} we deduce that  $ \overline{G}=G/F(G) $ is torsion, and, in particular, locally finite.

The argument above shows that  it is enough to prove the theorem under the additional assumption that $G$ is pronilpotent.  
Choose now  $a\in A^{\#}$. For each  pair  $i,j$ of positive integers we set $$S_{ij}=\{(x,y)\in C_G(a)\times C_G(a):[x,_iy^j]=1\}.$$ With an argument similar to that used in the proof of Theorem \ref{main}, with only obvious changes,  we can  show that $G=P_1\times\ldots\times P_r\times R$, where $R$ is a locally nilpotent subgroup of $G$  and  $P_i$ are finitely many Sylow subgroups of $G$. Therefore it is sufficient to show that each subgroup $P_i$ is locally virtually nilpotent.  This follows from Proposition \ref{q^3pro-p} and  the proof is complete.
\end{proof}

We conclude observing that  the proof of Corollary \ref{coro2} is analogous to that of Corollary \ref{coro1} and can be obtained, with obvious changes, by replacing every appeal to Theorem \ref{main} in the proof of Corollary \ref{coro1} by an appeal to Theorem \ref{main2}. Therefore we omit the further details.



\end{document}